\documentclass[12pt]{article}
\usepackage[utf8]{inputenc}
\usepackage{amsmath}
\usepackage{amssymb}
\usepackage{amsthm}
\usepackage{cite}
\usepackage{hyperref}
\usepackage{enumerate}
\usepackage[left=0.7in,right=0.7in,top=0.7in,bottom=0.7in]{geometry}
\usepackage{titlesec}
\titleformat*{\section}{\large\bfseries}
\newtheorem{theorem}{Theorem}[section]
\newtheorem{lemma}[theorem]{Lemma}
\newtheorem{corollary}[theorem]{Corollary}

\newtheorem{example}[theorem]{Example}

\newtheorem{remark}[theorem]{Remark}
\title{Quasi-Fredholm spectrum and compact perturbations}
\author{\large Anuradha Gupta and Ankit Kumar$^*$}
\date{}
\begin{document}
\maketitle
\begin{abstract}
In this paper we explore some characteristics of the quasi-Fredholm resolvent set $\rho_{qf}(T)$ of an operator $T$ defined on an infinite dimensional  Banach space $X$. Moreover, in the case of Hilbert space $H$, we study the stability of the SVEP and describe the operators for which the SVEP is preserved under compact perturbations  using quasi-Fredholm spectrum and $\rho_{qf}(T)$.


\end{abstract}
 \section{Introduction and Preliminaries}
 Throughout this paper, denote by $B(X)$ the Banach algebra of all bounded linear operators defined on an infinite dimensional complex  Banach space $X$. For $A \subset \mathbb{C},$ $\mbox{iso} \thinspace A$, $\mbox{int} \thinspace A$,  $\overline{A}$  and $\mbox{acc} \thinspace A$ denote the set of isolated points of $A$, interior points of $A$, closure of $A$ and accumulation points of $A$, respectively. For $\lambda\in \mathbb{C}$ and $r >0$, $B(\lambda, r)$ denotes the open disc of radius $r$ centred at $\lambda.$ For $T \in B(X)$, the null space of $T$, range of $T$, spectrum of $T$ and adjoint of $T$ are denoted by $N(T)$, $T(X)$, $\sigma(T)$ and $T^*$, respectively.
Let $\alpha(T)=$ dim $N(T)$ and $\beta(T)=$ codim $T(X)$ be the nullity of $T$ and deficiency of $T,$ respectively. An operator  $T \in B(X)$ is called bounded below if $T$ is injective and $T(X)$ is closed.  A bounded linear operator $T$ is said to be an upper semi-Fredholm operator if $\alpha(T) < \infty $ and $T(X)$ is closed.  An operator $T \in B(X)$ is said to be  a lower semi-Fredholm operator if $\beta(T) < \infty $. An operator $T \in B(X)$ is called a semi-Fredholm operator if it is either upper semi-Fredholm or lower semi-Fredholm. For a semi-Fredholm operator $T$, the index of $T$ is defined by $\mbox{ind}(T):= \alpha(T)-\beta(T)$. 
The \emph{point spectrum}, \emph{approximate point spectrum} and  \emph{semi-Fredholm  spectrum} are defined by  \begin{align*}
\sigma_p(T):&= \{ \lambda \in \mathbb{C}: \lambda I-T \thinspace \mbox{is not injective} \},\\
\sigma_a(T):&=\{\lambda \in \mathbb{C}: \lambda I-T \thinspace \mbox{is not bounded below}\},\\\sigma_{sf}(T):&=\{\lambda \in \mathbb{C}:   \lambda I-T \thinspace \mbox{is not semi-Fredholm}\}, \thinspace  \mbox{respectively.}
\end{align*}
Clearly, $\sigma_{sf}(T) \subset \sigma_{a}(T)$. Let $\rho_a(T)= \mathbb{C} \setminus \sigma_a(T)$ and $\rho_{sf}(T)= \mathbb{C} \setminus \sigma_{sf}(T)$. An operator $T \in B(X)$ is called said to be an upper semi-Weyl (lower semi-Weyl, respectively) operator  if it is upper semi-Fredholm (lower semi-Fredholm, respectively) and 
 $\mbox{ind}(T) \leq 0$ ($\mbox{ind}(T) \geq 0$, respectively). A bounded linear operator $T$ is called Weyl if it is semi-Fredholm and $\mbox{ind}(T)=0$. The \emph{Weyl essential approximate point  spectrum} and \emph{Weyl spectrum} are defined by
\begin{align*}
\sigma_{uw}(T):&= \{ \lambda \in \mathbb{C}: \lambda I -T \thinspace \mbox{is not upper semi-Weyl} \},\\
\sigma_w(T):&= \{ \lambda \in \mathbb{C}:\lambda I-T \thinspace \mbox{is not Weyl} \}, \thinspace \mbox{respectively.} 
\end{align*}
Let $\rho_{uw}(T)= \mathbb{C} \setminus \sigma_{uw}(T)$ and $\rho_{w}(T)= \mathbb{C} \setminus \sigma_{bw}(T)$. 
Let  $T \in B(X)$, then for each non negative integer $n$, $T$ induces a linear transformation $$ \Psi_{n}: T^n(X)/T^{n+1}(X)\longrightarrow T^{n+1}(X)/T^{n+2}(X)$$ defined by 
$$\Psi_{n}(y+T^{n+1}(X)):=T y+T^{n+1}(X), \quad y\in T^{n}(X).$$ Clearly, $\Psi_n$ is surjective for each $n$. For each $n$, let $k_{n}(T)=\alpha(\Psi_{n})$. Define a norm $\Vert . \Vert_{n}$ on $T^n(X)$ by 
$$\Vert y \Vert_{n}:=\inf\limits_{x \in X}\{\Vert x \Vert:y=T^n x\}.$$
The topology induced by this norm is called operator range topology on $T^n(X)$. 
An operator $T \in B(X)$ is said to have uniform descent for $ n \geq d$ if there exists a non  negative integer $d$ such that $k_{n}(T)=0$ for $n \geq d$. In addition, if $T^{n}(X)$ is closed in the operator range topology of $T^d(X)$ for $n \geq d$, then $T$ is said to have topological uniform descent for $n \geq d$. The \emph{topological uniform descent spectrum} is defined by $$\sigma_{\Gamma}(T):=\{\lambda \in \mathbb{C}: \lambda I-T \thinspace \mbox{does not have topological uniform descent}\}.$$
Let $\rho_{\Gamma}(T)=\mathbb{C} \setminus \sigma_{\Gamma}(T)$ be the topological uniform descent resolvent of $T$.

  For $T \in B(X)$ consider the set $$\Delta(T):=\{n\in \mathbb{N} :m \geq n, m\in \mathbb{N}\thinspace \thinspace \mbox{implies that} \thinspace  \thinspace T^n(X) \cap N(T) \subset T^m(X) \cap N(T)\}.$$
The degree of stable iteration is defined by $\mbox{dis}(T):= \inf \Delta(T)$ whenever $\Delta(T) \neq \emptyset$.  If $\Delta(T)=\emptyset$, set  $\mbox{dis}(T)= \infty$.
Let $T \in B(X)$. An operator $T \in B(X)$  is said to be quasi-Fredholm of degree $d$ if there exists a $d \in \mathbb{N}$ such that

(i) $\mbox{dis}(T)=d$,

(ii) $T^{n}(X)$ is a closed subspace of $X$ for each $n \geq d$,

(iii) $T(X)+ N(T^d)$ is a closed subspace of $X$.
 
For $T \in B(X)$, the \emph{quasi-Fredholm spectrum} is defined by $$\sigma_{qf}(T):=\{\lambda \in \mathbb{C}: \lambda I-T \thinspace \mbox{is not quasi-Fredholm}\}.$$ Let $\rho_{qf}(T)= \mathbb{C} \setminus \sigma_{qf}(T)$ be the quasi-Fredholm resolvent of $T$. By \cite[Theorem 1.96]{1} we know that $\sigma_{\Gamma}(T) \subset \sigma_{qf}(T) \subset \sigma_{sf}(T)$.
For a bounded linear operator $T$  and a non negative integer $n$, denote by $T_{[n]}$  the restriction of $T$ to $T^n(X)$. An opeartor $T \in B(X)$ is said to be B-Fredholm (an upper semi B-Fredholm, a lower semi B-Fredholm, respectively) if for some non negative integer $n$, $T^n(X)$ is closed and $T_{[n]}$ is Fredholm (an upper semi B-Fredholm, a lower semi B-Fredholm, respectively). In this case, the $\mbox{ind}(T)$ is defined to be the $\mbox{ind}T_{[n]}$ (see \cite{2}).  An operator is said to be a semi B-Fredholm operator if is a lower semi B-Fredholm or an upper semi B-Fredholm operator. The \emph{ semi B-Fredholm spectrum} is defined by  $$\sigma_{sbf}(T):=\{\lambda \in \mathbb{C}: \lambda I-T \thinspace \mbox{is not  semi B-Fredholm}\}.$$
Clearly, $\sigma_{sbf}(T) \subset \sigma_{sf}(T)$. Let $\rho_{sbf}(T)= \mathbb{C} \setminus \sigma_{sbf}(T)$.  By \cite[Theorem 1.116]{1} we know that every semi B-Fredholm operator is quasi-Fredholm. Therefore, $\sigma_{\Gamma}(T) \subset \sigma_{qf}(T) \subset \sigma_{sbf}(T) \subset \sigma_{sf}(T)$.
 An operator $T \in B(X)$ is called an upper semi B-Weyl (B-Weyl, respectively) operator if $T$ is an upper semi B-Fredholm (B-Fredholm, respectively) having $\mbox{ind}(T) \leq 0$ ($\mbox{ind}(T)=0$, respectively). The \emph{upper semi B-Weyl spectrum} and \emph{B-Weyl spectrum} are defined by  \begin{align*}
 \sigma_{usbw}(T):&=\{\lambda \in \mathbb{C}: \lambda I-T \thinspace \mbox{is not upper semi B-Weyl}\},\\
 \sigma_{bw}(T):&=\{\lambda \in \mathbb{C}: \lambda I-T \thinspace \mbox{is not B-Weyl}\} \thinspace \mbox{respectively}.
 \end{align*}
 
Let $\rho_{bw}(T)= \mathbb{C} \setminus \sigma_{bw}(T)$ and $\rho_{w}(T)= \mathbb{C} \setminus \sigma_{w}(T)$. For an operator $T \in B(X)$, the ascent of $T$ denoted by $p(T)$ is the smallest non negative integer $p$ such that $N(T^{p})= N(T^{p+1})$. If  no such integer exists, set $p(T)=\infty$. For an operator $T \in B(X)$, the descent of $T$ denoted by $q(T)$ is the smallest non negative integer $q$ such that $T^{q}(X)=T^{q+1}(X)$. If  no such integer exists,  set $q(T)=\infty$. Evidently,  $p(T)=0$ if and only if $T$ is injective and $q(T)=0$ if and only if $T$ is surjective. By \cite[Theorem 1.20]{1} we know that if both $p(T)$ and $q(T)$ are finite, then $p(T)=$ $q(T)$.   An operator $T \in B(X)$ is called left Drazin invertible if $p(T)<\infty$ and $T^{p+1}(X)$ is closed. We say that $\lambda \in \mbox{iso} \thinspace \sigma_{a}(T)$ is a  left pole of the resolvent of  $T$ if $\lambda I-T$ is left Drazin invertible. An operator $T \in B(X)$ is called right Drazin invertible if $q(T)<\infty$ and $T^{q}(X)$ is closed. An operator $T \in \mathcal{B}(X)$ is called Drazin invertible if $p(T)=q(T)< \infty$. We say that $\lambda \in \mbox{iso} \thinspace \sigma(T)$ is a  pole of the resolvent of $T$ if $\lambda I-T$ is Drazin invertible. The \emph{left Drazin spectrum} and \emph{Drazin spectrum} are defined by  
\begin{align*}
\sigma_{ld}(T):&=\{\lambda \in \mathbb{C}: \lambda I-T \thinspace \mbox{is not left Drazin invertible}\},\\
\sigma_{d}(T):&=\{\lambda \in \mathbb{C}: \lambda I-T \thinspace \mbox{is not  Drazin invertible}\}, \thinspace \mbox{respectively}.
\end{align*}
By \cite[Theorem 1.142]{1} we know that $\sigma_{qf}(T) \subset \sigma_{ld}(T) \subset \sigma_d(T)$. The set of all the poles of the resolvent of $T$ and  all left poles of the resolvent of  $T$ are denoted by $\Pi(T)=\sigma(T) \setminus \sigma_{d}(T)$ and $\Pi^{a}(T)=\sigma_a(T) \setminus \sigma_{ld}(T)$, respectively.

An operator $T \in B(X)$ is said to have the single-valued extension property (SVEP) at $\lambda_{0} \in \mathbb{C}$, if for every neighborhood $V$ of $\lambda_{0}$ the only analytic function $f:V \rightarrow X $ which satisfies the equation $(\lambda I-T)f(\lambda)=0$ is the function $f=0$. An operator $T \in B(X)$ is said to have SVEP if $T$ has SVEP at every $\lambda \in \mathbb{C}$. It is known that if $\mbox{int} \thinspace \sigma_p(T)= \emptyset$, then $T$ has SVEP. Recall that $$p(\lambda I-T)<\infty \thinspace \thinspace  \mbox{implies that} \thinspace \thinspace  T \thinspace \mbox{has SVEP at}\thinspace\thinspace \lambda$$ and $$q(\lambda I-T)<\infty \thinspace \thinspace  \mbox{implies that} \thinspace \thinspace  T^* \thinspace \mbox{has SVEP at}\thinspace\thinspace \lambda.$$ 

 Zeng et al. \cite{5} studied the components of quasi-Fredholm resolvent and characterized them by means of localized SVEP. Shi  \cite{4} considered the topological uniform descent and studied how topological uniform descent resolvent is distributed in $\rho_{sf}(T)$. As we know that for an operator $T \in B(X)$, topological uniform descent, quasi-Fredholmness, semi-Fredholmness and semi B-Fredholness are closely related to each other. Motivated by them  we study the distribution of $\rho_{qf}(T)$ in $\rho_{sbf}(T)$. Zhu and Li \cite{6} obtained results for non commuting compact perturbations of an operator $T\in B(X)$ using semi-Fredholm spectrum. Recently, for $T \in B(X)$ various authors (see \cite{9,10,11}) discussed various spectral properties under compact (not necessarily commuting) perturbations. Motivated by them we obtain  results for  compact perturbations of an operator $T\in B(X)$ using quasi-Fredholm spectrum.

In this paper we discuss some characteristics of quasi-Fredholm resolvent set $\rho_{qf}(T)$ for $T\in B(X)$. We give results regarding the distribution of semi B-Fredholm domain $\rho_{sbf}(T)$ in $\rho_{qf}(T)$. We prove that if $ \mbox{int} \thinspace  \sigma_{sbf}(T) = \emptyset$, then there  is one-to-one correspondence between the bounded components of $\rho_{sbf}(T)$ and the  bounded components of $\rho_{qf}(T)$. In the last section we discuss the permanence of SVEP under(small) compact perturbations using quasi-Fredholm resolvent set and quasi-Fredholm spectrum. Also, we describe those operators for which SVEP is stable under compact perturbations by  means of quasi-Fredholm resolvent.
 \section{Main Results}
  It is known that the sets $\rho_{sf}(T), \rho_{sbf}(T),\rho_{qf}(T)$ and $\rho_{sf}(T)$ are nonempty open sets of $\mathbb{C},$ they can be decomposed into (pairwise disjoint, maximal, open, connected) non-empty components. 
\begin{lemma}
 Let $T \in B(X)$. Then $\sigma_{sf}(T)=\sigma_{sbf}(T) \cup  \mbox{iso} \thinspace  \sigma_{sf}(T)$.
\end{lemma}
\begin{proof}
Let $\lambda_{0} \in \sigma_{sf}(T) \setminus \sigma_{sbf}(T)$. Then $\lambda_{0} I-T$ is semi B-Fredholm. By \cite[Theorem 1.117]{1} there exists an $\epsilon >0$ such that $\lambda I-T$ is semi-Fredholm for all $\lambda \in B(\lambda_{0}, \epsilon) \setminus \{\lambda_0\}.$ Therefore, $\lambda_{0} \in  \mbox{iso} \thinspace  \sigma_{sf}(T)$. Thus,  $\sigma_{sf}(T) \subset \sigma_{sbf}(T) \cup  \mbox{iso} \thinspace  \sigma_{sf}(T)$. The reverse inclusion  always holds.
\end{proof}

Recall that  a hole of a compact set $\sigma \subset \mathbb{C}$ is a bounded component of $\mathbb{C} \setminus \sigma$. It is known that  $\mathbb{C} \setminus \sigma$ has always an unbounded component. Therefore,  $\mathbb{C} \setminus \sigma$ is connected if and only if $\sigma$ has no holes.
\begin{theorem}\label{theorem1}
Let $T \in B(X)$, then $\rho_{sbf}(T)$ is connected if and only if $\rho_{sf}(T)$ is connected.
\end{theorem}
\begin{proof}
Suppose that $\rho_{sbf}(T)$ is connected. Since $\sigma_{sf}(T)=\sigma_{sbf}(T) \cup \mbox{iso} \thinspace  \sigma_{sf}(T)$, $\rho_{sf}(T)=\rho_{sbf}(T) \setminus \mbox{iso} \thinspace \sigma_{sf}(T)$. As $\rho_{sbf}(T)$ is connected and $\mbox{iso} \thinspace \sigma_{sf}(T)$ is at most countable  we deduce that $\rho_{sf}(T) $ is connected.

Conversely, suppose that $\rho_{sf}(T)$ is connected. Assume that $\rho_{sbf}(T)$ is not connected then there exists a bounded component $\Omega$ of $\rho_{sbf}(T)$. Then either $\Omega \cap \rho_{sf}(T) = \emptyset$ or $\Omega \cap \rho_{sf}(T) \neq \emptyset $. If  $\Omega$ $\cap$ $ \rho_{sf}(T) = \emptyset $, then $\Omega \subset \sigma_{sf}(T)$ which implies that  $\Omega \subset \mbox{iso} \thinspace \sigma_{sf}(T)$ which is not possible. Therefore, $\Omega \cap \rho_{sf}(T) \neq \emptyset $. Then there exists  $\lambda_0$ such that $\lambda_0 \in \Omega \cap \rho_{sf}(T)$. Let $\Omega ^{'}$ be the component of $\rho_{sf}(T)$ containing $\lambda_0$. Therefore, $\Omega ^{'}$ is an open connected subset of $\rho_{sbf}(T)$ such that $\Omega \cap \Omega ^{'} \neq \emptyset$. This implies that $\Omega ^{'} \subset \Omega$. Thus, $\Omega^{'}$ is a bounded component of  $\rho_{sf}(T)$, a contradiction. Hence, $\rho_{sbf}(T)$ is connected.
\end{proof}
By \cite[Lemma 2.2]{3} we have $\sigma_{uw}(T)=\sigma_{usbw}(T) \cup \mbox{iso} \thinspace \sigma_{uw}(T)$ and $\sigma_{w}(T)=\sigma_{bw}(T) \cup \mbox{iso} \thinspace \sigma_{w}(T)$. Following the lines of the proof of Theorem \ref{theorem1} we have the following result:
\begin{theorem}
Let $T \in B(X)$, then

(i) $\rho_{bw}(T)$ is connected if and only if $\rho_{w}(T)$ is connected.

(ii) $\rho_{usbw}(T)$ is connected if and only if $\rho_{uw}(T)$ is connected.
\end{theorem}
\begin{theorem}\label{theorem2}
Let $T \in B(X)$ and $\Omega _{qf}$ be a connected component of $\rho_{qf}(T)$. If $\Omega _{qf} \cap \rho_{sbf}(T) \neq \emptyset$, then there exists a  unique connected component $\Omega _{sbf}$ of $\rho_{sbf}(T)$ such that $\Omega _{qf}=\Omega _{sbf} \cup E$, where $E_0 \subset \mbox{iso} \thinspace \sigma_{sf}(T)$.
\end{theorem}
\begin{proof}
As $\Omega_{qf}$ is a connected component of $\rho_{qf}(T)$ and $\rho_{qf}(T) \subset \rho_{\Gamma}(T)$, there exists a component $\Omega_{\Gamma}$ of $\rho_{\Gamma}(T)$ such that $\Omega_{qf} \subset \Omega_{\Gamma}$. Since $\Omega _{qf} \cap \rho_{sbf}(T) \neq \emptyset$ and $\rho_{sbf}(T) \subset \rho_{qf}(T)$, proceeding as in the proof of Theorem \ref{theorem1} there exists a component $\Omega _{sbf}$ of $\rho_{sbf}(T)$ such that  $\Omega _{sbf} \subset  \Omega _{qf}$. By the proof of  Theorem \ref{theorem1} we get a component  $\Omega _{sf}$ of $\rho_{sf}$ such that  $\Omega _{sf} \subset \Omega _{sbf} \subset \Omega _{qf} \subset \Omega_{\Gamma}$. Using \cite[Theorem 1]{4} we have  $\Omega _{\Gamma}= \Omega _{sf} \cup E$, where $E \subset \mbox{iso} \thinspace \sigma_{sf}(T)$.  This gives $\Omega _{sbf} \subset \Omega_{qf} \subset \Omega_{sf} \cup E \subset \Omega _{sbf} \cup E$. Therefore, there exists  $E_0 \subset E \subset \mbox{iso} \thinspace \sigma_{sf}(T)$ such that  $\Omega _{qf}=\Omega _{sbf} \cup E_0$.

Assume that there exist connected components  $\Omega _{sbf}$ and $\Omega _{sbf} ^{'}$ of $\rho_{sbf}(T)$  such that $\Omega _{qf}=\Omega _{sbf} \cup E$ and  $\Omega _{qf}=\Omega _{sbf} ^{'} \cup F$, where $E, F \subset \mbox{iso} \thinspace \sigma_{sf}(T)$. Then $\Omega _{sbf} \cup E = \Omega _{sbf} ^{'} \cup F$ which implies that $\Omega _{sbf} \subset F$, a contradiction.
\end{proof}
\begin{corollary}\label{corollary1}
Let $T \in B(X)$ and $\Omega_{sbf}$  be a connected component of $\rho_{sbf}(T)$. Then there exists a  unique connected component $\Omega_{qf}$ of $\rho_{qf}(T)$ such that  $\Omega_{qf}= \Omega_{sbf} \cup E, $ where $E \subset \mbox{iso} \thinspace  \sigma_{sf}(T).$
\end{corollary}
\begin{proof}
Since $\Omega_{sbf} \subset \rho_{sbf}(T) \subset \rho_{qf}(T),$ there exists a connected component $\Omega_{qf}$ of $\rho_{qf}(T)$  such that $\Omega_{sbf} \subset \Omega_{qf}$. By Theorem \ref{theorem2}  we get $\Omega_{qf}=\Omega_{sbf} \cup E_0$, where $E_0 \subset \mbox{iso} \thinspace \sigma_{sf}(T)$. Assume that there exists another connected component $\Omega_{qf} ^{'}$ of $\rho_{qf}(T)$ such that $\Omega_{qf} ^{'} =\Omega_{sbf} \cup F$, where $F \subset \mbox{iso} \thinspace \sigma_{sf}(T)$. This gives $\Omega_{sbf} \subset \Omega_{qf} \cap \Omega_{qf} ^{'}$, a contradiction.
\end{proof}
\begin{corollary} \label{corollary2}
Let $T \in B(X)$ and $\mbox{int} \thinspace \sigma_{sbf}(T) = \emptyset. $ Then $\rho_{qf}(T) \setminus \rho_{sbf}(T)$ is at most countable and $\sigma_{sbf}(T)=\sigma_{qf}(T) \cup \mbox{iso} \thinspace  \sigma_{sbf}(T)$.
\end{corollary}
\begin{proof}
Let $\{\Omega_{qf} ^{n}\}_{n=1}^\infty$ be an enumeration of connected components of $\rho_{qf}(T)$. Since $\mbox{int} \thinspace  \sigma_{sbf}(T) = \emptyset$, for every connected component $\Omega_{qf} ^{n}$ of $\rho_{qf}(T)$ we have $\Omega_{qf} ^{n} \cap \rho_{sbf}(T) \neq \emptyset$. Using Theorem \ref{theorem2} for $\Omega_{qf} ^{n}$, there exists a unique connected component $\Omega_{sbf} ^{n}$ of $\rho_{sbf}(T)$ such that $\Omega_{qf} ^{n}$=$\Omega_{sbf} ^{n} \cup E^{n},$ where $E^{n} \subset \mbox{iso} \thinspace \sigma_{sf}(T)$. Let $E= \bigcup\limits_{n=1}^\infty E^{n}$, then $E$ is  at most countable and $E \subset \mbox{iso} \thinspace \sigma_{sf}(T)$. Also,
$$ \rho_{qf}(T)=\bigcup_{n=1}^\infty \Omega_{qf} ^{n}=\bigcup_{n=1}^\infty \Omega_{sbf} ^{n}  \cup E.$$ Since $\rho_{sbf}(T) \subset \rho_{qf}(T)$, $\rho_{qf}(T)=\rho_{sbf}(T) \cup E$. Let $E^{'}= E \cap \sigma_{sbf}(T)$.  Then $E^{'} \subset \mbox{iso} \thinspace \sigma_{sbf}(T)$ and $\rho_{qf}(T)=\rho_{sbf}(T) \cup E^{'}$.  This gives $\sigma_{sbf}(T)=\sigma_{qf}(T) \cup E^{'}$ which implies that $\sigma_{sbf}(T)=\sigma_{qf}(T) \cup \mbox{iso} \thinspace \sigma_{sbf}(T)$.
\end{proof}
Let $W_1$, $W_2$, $W_3$ and $W_4$ be the set of all bounded components of $\rho_{sbf}(T)$, $\rho_{qf}(T), \rho_{sf}(T)$ and $\rho_{\Gamma}(T)$, respectively.
\begin{theorem}\label{theorem3}
Let $T \in B(X)$. Then there exists an injective mapping $f:W_1 \rightarrow W_2$. Moreover, if $\mbox{int} \thinspace \sigma_{sbf}(T)=\emptyset$, then $f$ is also surjective.
\end{theorem}
\begin{proof}
Suppose that $\Omega \in W_1$. Using Corollary \ref{corollary1} we get a unique connected component $\Omega^{'}$ of $\rho_{qf}(T)$ such that  $\Omega^{'}= \Omega \cup E$, where $E \subset \mbox{iso} \thinspace \sigma_{sf}(T)$. Since $\mbox{iso} \thinspace \sigma_{sf}(T) \subset \sigma_{sf}(T)$,  $\Omega^{'}$ is bounded component of $\rho_{qf}(T)$ which implies that $\Omega^{'} \in W_2$. Define $f:W_1  \rightarrow W_2$ by $f(\Omega)=\Omega^{'}$. Then $f$ is a well defined mapping. We prove that $f$ is an  injective mapping. Let $\Omega_1$ and $\Omega_2$ be two distinct elements of $W_1$ such that $f(\Omega_1)=f(\Omega_2)$. This implies that there exists a component $\Omega^{'}$ of $\rho_{qf}(T)$ such that $\Omega^{'}= \Omega_1 \cup E =\Omega_2 \cup F$, where $E, F \subset  \mbox{iso} \thinspace \sigma_{sf}(T)$. As $\Omega_1 \cap \Omega_2 = \emptyset$, $\Omega_1 \subset  F$, a contradiction. Therefore, $f$ is an injective mapping.

Suppose that $\tau \in W_2$. Since $\mbox{int} \thinspace  \sigma_{sbf}(T)=\emptyset$, $\tau \cap \rho_{sbf}(T) \neq \emptyset$. Using Theorem \ref{theorem2} there exists a unique component $\tau^{'}$ of $\rho_{sbf}(T)$ such that $\tau=\tau^{'} \cup E$, where $E \subset \mbox{iso} \thinspace \sigma_{sf}(T)$. Therefore, $f( \tau^{'})=\tau$.
 \end{proof}
  Similarly, using \cite[Theorem 1, Corollary 1]{4} we establish the following result: 
  \begin{theorem}\label{theorem4}
Let $T \in B(X)$. Then there exists an injective mapping $g:W_3 \rightarrow W_4$. Moreover, if $\mbox{int} \thinspace \sigma_{sf}(T)=\emptyset$, then $g$ is also surjective.
\end{theorem}
 \begin{theorem}
 Let $T \in B(X)$. Then every non isolated boundary point of $\sigma_{sbf}(T)$ belongs to $\sigma_{qf}(T)$.
 \end{theorem}
 \begin{proof}
 Let $\lambda$ be a non isolated boundary point of $\sigma_{sbf}(T)$. Let $\lambda \in \rho_{qf}(T)$ and $\Omega_{qf}$ be the component of $\rho_{qf}(T)$ containing $\lambda$.  Then there exists $\epsilon > 0$ such that $B(\lambda, \epsilon) \subset \Omega_{qf}$. Since $\lambda$ is the boundary point of $\sigma_{sbf}(T)$, $B(\lambda, \epsilon) \cap \rho_{sbf}(T) \neq \emptyset$ which implies that $\Omega_{qf} \cap \rho_{sbf}(T) \neq \emptyset$. Therefore, by Theorem \ref{theorem2} there exists a component $\Omega_{sbf}$ of $\rho_{sbf}(T)$ such that $\Omega_{qf}=\Omega_{sbf} \cup E$, where $E \subset \mbox{iso} \thinspace \sigma_{sf}(T)$. Since $\lambda \in \Omega_{qf} \cap \mbox{acc} \thinspace \sigma_{sbf}(T) \subset \Omega_{qf} \cap \mbox{acc} \thinspace \sigma_{sf}(T) $ we deduce that $\lambda \in \Omega_{sbf}$, a contradiction. Therefore, $\lambda \in \sigma_{qf}(T)$.
 \end{proof}
 \begin{remark}\label{remark2}
 \emph{It is observed that if $P$ is a closed subset of $\mathbb{C}$ such that $\mbox{int} \thinspace P \neq \emptyset$ and $\mbox{int} \thinspace P^c \neq \emptyset$, then $(\partial P)^c$ is disconnected.}
 \end{remark}
 \begin{lemma}\label{lemma1}
 Let $T \in B(X)$, $\rho_{qf}(T)$ be connected and $\mbox{int} \thinspace \sigma_{qf}(T)=\emptyset$. Suppose that $P$ is a  closed set contained in $\sigma(T)$. Then  $\mbox{int} \thinspace P=\emptyset$.
 \end{lemma}
 \begin{proof}
 Suppose that $\mbox{int} \thinspace P \neq \emptyset$. First we prove that $\rho_{qf}(T) \cap \mbox{acc} \thinspace (\partial P) \neq \emptyset$. If $\rho_{qf}(T) \cap \mbox{acc} \thinspace (\partial P)= \emptyset$, then $$\rho_{qf}(T) \subset \mbox{iso} \thinspace (\partial P)\cup  (\partial P)^c \subset \overline{\rho_{qf}(T)}=\mathbb{C}.$$ Since $\rho_{qf}(T)$ is connected, $\mbox{iso} \thinspace (\partial P) \cup  (\partial P)^c$ is connected. Let $S=\mbox{iso} \thinspace (\partial P) \cup  (\partial P)^c.$ This implies that $(\partial P)^c= S \setminus \mbox{iso} \thinspace (\partial P)$ which gives $(\partial P)^c$ is connected. As $ \mbox{int} \thinspace P \neq \emptyset$ then by Remark \ref{remark2} we get a contradiction. Therefore, there exists $\lambda$ such that $\lambda \in \rho_{qf}(T) \cap \mbox{acc} \thinspace (\partial P)$. As $\rho_{qf}(T)$ is connected and $\rho(T) \subset \rho_{qf}(T)$, by \cite[Theorems  3.6, 3.7]{5} $p(\lambda I-T)=q(\lambda I-T) < \infty$ for all  $\lambda \in \rho_{qf}(T)$. Therefore, $(\lambda I-T)$ is drazin invertible for all  $\lambda \in \rho_{qf}(T)$. This gives $\rho_{qf}(T)=\rho(T)\cup \Pi(T)$, where $\Pi(T)$ denotes the set of  poles of the resolvent of $T$. Since $\lambda \in \mbox{acc} \thinspace (\partial P) \subset \sigma(T)$ which implies that $\lambda \in \Pi(T) \subset \mbox{iso} \thinspace \sigma(T)$. Then there exists an $\epsilon > 0$ such that $B(\lambda, \epsilon) \setminus \{\lambda\} \subset \rho(T)$. Since $\lambda \in \mbox{acc} \thinspace (\partial P)$, there exists $\mu \in B(\lambda, \epsilon) \cap \partial P \subset \rho(T) \cap \partial P$,  a contradiction. Hence,  $\mbox{int} \thinspace P =\emptyset$.
 \end{proof}
 If $\rho_{\Gamma}(T)$ is connected, then by \cite[Proposition 2]{4}  we know that $\rho_{\Gamma}(T)= \rho(T) \cup \Pi(T)$. Then proceeding likewise as in Lemma \ref{lemma1} we have the following result:
 \begin{lemma}\label{lemma2}
 Let $T \in B(X)$,  $\rho_{\Gamma}(T)$ be connected and $\mbox{int} \thinspace \sigma_{\Gamma}(T)=\emptyset$. Suppose that $P$ is closed set contained in $\sigma(T)$. Then  $\mbox{int} \thinspace P=\emptyset$.
 \end{lemma}
\begin{theorem}\label{theorem8}
Let $T \in B(X)$. Then following statements are equivalent:

(i) $\rho_{sf}(T)$ is connected and int $\sigma_{sf}(T)=\emptyset$,

(ii) $\rho_{sbf}(T)$ is connected and int $\sigma_{sbf}(T)=\emptyset$,

(iii) $\rho_{qf}(T)$ is connected and int $\sigma_{qf}(T)=\emptyset$,

(iv) $\rho_{\Gamma}(T)$ is connected and  $\mbox{int} \thinspace \sigma_{\Gamma}(T)=\emptyset$.
\end{theorem}
\begin{proof}
Since $\sigma_{\Gamma}(T) \subset \sigma_{qf}(T) \subset \sigma_{sbf}(T) \subset \sigma_{sf}(T)$, (i)$\Rightarrow$(ii)$\Rightarrow$(iii)$\Rightarrow$(iv) is obvious.
 Now suppose that $\rho_{\Gamma}(T)$ is connected and int $\sigma_{\Gamma}(T)=\emptyset$. Using Lemma \ref{lemma2} we deduce that $\mbox{int} \thinspace \sigma_{sf}(T)=\emptyset$. It remains to prove that $\rho_{sf}(T)$ is connected. As $\rho_{\Gamma}(T)$ is connected, by \cite[Theorem 1]{4} there exists a unique connected component $\Omega_{sf}$ of $\rho_{sf}(T)$ such that $\rho_{\Gamma}(T)= \Omega_{sf} \cup E_0$, where $E_0 \subset \mbox{iso} \thinspace \sigma_{sf}(T)$. This gives $\rho_{\Gamma}(T)=\rho_{sf}(T) \cup E_0$. This implies that $rho_{sf}(T)=\rho_{\Gamma}(T) \setminus E_0$. Therefore, $\rho_{sf}(T)$ is connected.
\end{proof}
\begin{lemma}\label{lemma4}
Let $T \in B(X)$. Then if $\rho_{sf}(T)$ consists of finite bounded components, then $\rho_{sbf}(T)$ consists of finite bounded components.
\end{lemma}
\begin{proof}
Suppose that $\Omega_1$ and $\Omega_2$ are two distinct bounded components of  $\rho_{sbf}(T)$. Then by the proof of Theorem \ref{theorem1} we get bounded components $\Omega_{1} ^{'}$, $\Omega_{2} ^{'}$ of $\rho_{sf}(T)$ such that $\Omega_{1} ^{'} \subset \Omega_1$ and $\Omega_{2} ^{'} \subset \Omega_2$. This gives $\Omega_{1} ^{'} \cap \Omega_{2} ^{'}=\emptyset$ since if $\Omega_{1} ^{'} = \Omega_{2} ^{'}$, then $\Omega_1 \cap \Omega_2 \neq \emptyset$ which is a contradiction.
\end{proof}
\begin{remark}\label{remark1}
\emph{If $\mbox{int} \thinspace \sigma_{qf}(T)=\emptyset$ and $\Omega$ is bounded component of $\rho_{\Gamma}(T)$, then $\Omega \cap \rho_{qf}(T) \neq \emptyset$. Therefore, there exists a component $\Omega^{'}$ of $\rho_{qf}(T)$ such that $\Omega^{'} \subset \Omega$. From this we can conclude that for any two bounded distinct components of  $\rho_{\Gamma}(T)$ we get two distinct component of $\rho_{qf}(T)$. Hence, if $\rho_{qf}(T)$ consists of finite bounded components, then $\rho_{\Gamma}(T)$ consists of bounded components.}
\end{remark} 
\begin{theorem}\label{theorem9}
Let $T \in B(H)$ and $int \sigma_{p}(T)=\emptyset$. Then  following statements are equivalent:

(i) $\mbox{int} \thinspace  \sigma_{sf}(T)=\emptyset$ and $\rho_{sf}(T)$ consists of finite bounded components,

(ii) $\mbox{int} \thinspace \sigma_{sbf}(T)=\emptyset$ and $\rho_{sbf}(T)$ consists of finite bounded components,

(iii) $\mbox{int} \thinspace \sigma_{qf}(T)=\emptyset$ and $\rho_{qf}(T)$ consists of finite bounded components,

(iv) $\mbox{int} \thinspace \sigma_{\Gamma}(T)=\emptyset$ and $\rho_{\Gamma}(T)$ consists of finite bounded components.
\end{theorem}
\begin{proof}
(i) $\Rightarrow$(ii) Follows directly from Lemma \ref{lemma4}.\\
(ii) $\Rightarrow$(iii)  Follows from Theorem \ref{theorem3}.\\
(iii) $\Rightarrow$(iv) Follows from Remark \ref{remark1}.\\
(iv) $\Rightarrow$(i) From Theorem \ref{theorem4} it follows  that $\rho_{sf}(T)$ consists of finite bounded components. Also, if $\mbox{int} \thinspace \sigma_{p}(T)=\emptyset$, then $T$ has SVEP. Therefore, by \cite[Proposition]{4} we have  $\mbox{int} \thinspace \sigma_{sf}(T)=\emptyset$.
\end{proof}
\section{Quasi-Fredholm spectrum and compact perturbations}
Let $K(X)$ denote the  ideal of all compact operators acting on a Banach space $X$. It is known that for $T \in B(X)$ and $K \in K(X)$, $$\sigma_{*}(T)=\sigma_{*}(T+K),$$ where $\sigma_{*}=\sigma_{sf}$ or $\sigma_{uw}$ or $\sigma_{w}$.  We start the following section with the following theorem:

\begin{theorem}\label{theorem11}
Let $T \in B(X)$ and $\rho_{qf}(T)$ be connected. Suppose that $K \in K(X)$.  Then $$\sigma(T+K)=\sigma_{qf}(T+K) \cup \Pi(T+K)\cup (\sigma_{sbf}(T+K) \setminus \sigma_{qf}(T+K)).$$ 
\end{theorem}
\begin{proof}
Since $\rho_{qf}(T)$ is connected, $\rho_{qf}(T)=\rho(T) \cup \Pi(T)$ which implies that $\rho_{qf}(T)=\rho_{sbf}(T)$. Therefore, $\rho_{sbf}(T)$ is connected. Using Theorem \ref{theorem1} we get $\rho_{sf}(T)$ is connected. As $\rho_{sf}(T)=\rho_{sf}(T+K)$, $\rho_{sf}(T+K)$ is connected. Again using Theorem \ref{theorem1} $\rho_{sbf}(T+K)$ is connected.
By Corollary \ref{corollary1} there exists a component $\Omega_{qf}$ of $\rho_{qf}(T+K)$ such that  $$\Omega_{qf}=\rho_{sbf}(T+K) \cup E,$$ where $E \subset \mbox{iso} \thinspace \sigma_{sf}(T+K)$. Let $E_0=E \cap  \sigma_{sbf}(T+K)$. Then $\Omega_{qf}=\rho_{sbf}(T+K) \cup E_0$  and  $E_0 \subset \mbox{iso} \thinspace \sigma_{sbf}(T+K)$. As $\rho(T+K) \subset \rho_{sbf}(T+K) \subset \Omega$, by \cite[Theorems  3.6, 3.7]{5} we get $\Omega= \rho(T+K) \cup \Pi(T+K)$. Now $$\rho_{qf}(T+K)=\Omega \cup (\rho_{qf}(T+K)\cap \Omega^c).$$
As $\rho_{qf}(T+K)\cap \Omega^c = \sigma_{sbf}(T+K) \setminus \sigma_{qf}(T+K) $, $$\rho_{qf}(T+K)=\rho(T+K) \cup \Pi(T+K) \cup ( \sigma_{sbf}(T+K) \setminus \sigma_{qf}(T+K))$$ which gives $\sigma(T+K)=\sigma_{qf}(T+K) \cup \Pi(T+K)\cup (\sigma_{sbf}(T+K) \setminus \sigma_{qf}(T+K))$. 
\end{proof}
Denote 
$$\rho_{sbf} ^{+}(T)=\{\lambda \in \rho_{sbf}:\mbox{ind}(\lambda I-T)>0\}$$ and 
$$\rho_{sf} ^{+}(T)= \{ \lambda \in \rho_{sf}: \mbox{ind}(\lambda I-T)>0\}.$$
\begin{lemma}\label{lemma3}
Let $T \in B(X).$ Then $\rho_{sbf} ^{+}(T)=\emptyset$ if and only if $\rho_{sf} ^{+}(T)=\emptyset$.
\end{lemma}
\begin{proof}
Evidently, if $\rho_{sbf} ^{+}(T)=\emptyset$, then  $\rho_{sf} ^{+}(T)=\emptyset.$ Conversely, suppose that $\rho_{sf} ^{+}(T)=\emptyset$. Without loss of generality, we may assume that $0 \in  \rho_{sbf} ^{+}(T)$ then $0 \in  \rho_{sbf}(T)$ and $\mbox{ind}(T)>0.$ Now by \cite[Theorem 1.117]{1}  there exists an $\epsilon >0$ such that $B(0,\epsilon) \setminus \{0\} \subset \rho_{sf}(T)$ and $\mbox{ind}(\lambda I-T)> 0$ for all $\lambda \in B(0,\epsilon)$,  a contradiction.
\end{proof}
The following result  is an immediate consequence of Lemma \ref{lemma3} and \cite[Theorem 1.1]{6}
\begin{theorem}\label{theorem5}
Let $T \in B(H)$, where $H$ is a Hilbert space. Then the following statements are equivalent:

(i) There exists $K \in K(H)$ such that $T+K$ has SVEP,

(ii) $\rho_{sbf} ^{+}(T)=\emptyset$,

(iii) $\rho_{sf} ^{+}(T)=\emptyset$.
\end{theorem}
It is observed that if $T$ has SVEP at every point of $\rho_{qf}(T)$ then  $\rho_{qf}(T)$ need not be connected. The following example illustrates this fact:
\begin{example}
\emph{Let $R$ be unilateral shift on $l^2(\mathbb{N})$. It is  known that $\sigma_{a}(T)= S^{1}$, where $S^{1}$ denotes the unit circle. Therefore, $\partial \sigma_{a}(T) \cap \mbox{acc} \thinspace  \sigma_{a}(T)=S^1$.  By \cite[Corollary 3.6]{7} we have  $\sigma_{qf}(T)=\sigma_{usbb}(T)=S^1$. Hence, $T$ has SVEP at every point of $\rho_{qf}(T)$  but $\rho_{qf}(T)$ is not connected.}
\end{example}
\begin{theorem}\label{theorem6}
Let $T \in B(H)$, where $H$ is a Hilbert space. Then $T+K$ has SVEP at every point of $\rho_{qf}(T+K)$ for any $K \in K(H)$ if and only if $\rho_{qf}(T+K)$ is connected for any $K \in K(H)$.
\end{theorem}
\begin{proof}
Suppose that $T+K$ has SVEP at every point of $\rho_{qf}(T+K)$ for any $K \in K(H)$. Let $\rho_{qf}(T+K)$ is not connected for some $K \in K(H)$. Then we can find a bounded connected component $\Omega$ of $\rho_{qf}(T+K)$. Now $\partial \Omega \subset \sigma_{qf}(T+K) \subset \sigma_{sf}(T+K)$. Then the proof of \cite[corollary 4]{4} shows that there exist compact operators $K_1$ and $K_2$ such that   \[ 
  T+K+K_1=\left( {\begin{array}{cc}
   A+K_2 & * \\
  0 & B \\
  \end{array} } \right)\left( \begin{array}{c}
  H_1\\
  H_1^{\perp}\\
  \end{array} \right)
  \]
where $A$ is a normal operator and  $\lambda I-(A+K_2)$ is Weyl but not invertible operator for any $\lambda \in \Omega$. Since $T+K$ has SVEP at every point of $\rho_{qf}(T+K)$, $T+K$ has SVEP at $\Omega$. Then  by \cite[Theorem 3.6]{5} we have $$\Omega \subset \rho_{a}(T+K) \cup \mbox{iso} \thinspace \sigma_{a}(T+K).$$ Therefore, $\Omega \cap \rho_{a}(T+K) \neq \emptyset$ which implies that $\Omega \cap \rho_{sf}(T+K) \neq \emptyset$.  Let $\lambda \in \Omega \cap \rho_{sf}(T+K)$ and $\Omega^{'}$ be a connected component of $\rho_{sf}(T+K)$ contaning $\lambda$. Then  $\Omega^{'} \subset \Omega$. Since $T+K+K_1$ has SVEP at every point of $\rho_{qf}(T+K+K_1)$,  $T+K+K_1$ has SVEP at every point of  $\rho_{sf}(T+K+K_1)$. Using \cite[Theorem 3.36]{8} we have
 $$\Omega^{'} \subset \rho_{a}(T+K+K_1)\cup \mbox{iso} \thinspace \sigma_{a}(T+K+K_1).$$  This gives $\Omega^{'} \cap \rho_a(T+K+K_1) \neq \emptyset$. Therefore, there exists $ \mu \in \Omega^{'} \subset \Omega$ such that $\mu I-(T+K+K_1)$ is bounded below which implies that $\mu I-(A+K_2)$ is bounded below, a contradiction.
 Conversely, suppose that $\rho_{qf}(T+K)$ is connected for any $K \in K(H)$. Then by \cite[Theorem 3.6]{5}, we know that $T+K$ has SVEP at every point of $\rho_{qf}(T+K)$.
\end{proof}
The following result follows from \cite[Proposition 6, Corollary 4]{4} and  Theorem \ref{theorem8}.
\begin{theorem}\label{theorem7}
Let $T \in B(H)$, where $H$ is a Hilbert space. Then following statements are equivalent:

(i) $T+K$ has SVEP for any $K \in K(H)$,

(ii) $T^{*} +K$ has SVEP for any $K \in K(H)$,

(iii) $\rho_{sf}(T)$ is connected and $\mbox{int} \thinspace \sigma_{sf}(T)=\emptyset$,

(iv) $\rho_{sbf}(T)$ is connected and $\mbox{int} \thinspace \sigma_{sbf}(T)=\emptyset$,

(v) $\rho_{qf}(T)$ is connected and $\mbox{int} \thinspace \sigma_{qf}(T)=\emptyset$,

(vi) $\rho_{\Gamma}(T)$ is connected and $\mbox{int} \thinspace \sigma_{\Gamma}(T)=\emptyset$,

(vii) $\rho_{sf}(T+K)$ is connected and $\mbox{int} \thinspace \sigma_{sf}(T+K)=\emptyset$ for any $K \in K(H)$,

(viii) $\rho_{sbf}(T+K)$ is connected and $\mbox{int} \thinspace \sigma_{sbf}(T+K)=\emptyset$ for any $K \in K(H)$,

(ix) $\rho_{qf}(T+K)$ is connected and $\mbox{int} \thinspace \sigma_{qf}(T+K)=\emptyset$ for any $K \in K(H)$,

(x) $\rho_{\Gamma}(T+K)$ is connected and $\mbox{int} \thinspace \sigma_{\Gamma}(T+K)=\emptyset$.
\end{theorem}
The following result is consequence of \cite[Theorem 1.2]{6} and Theorem \ref{theorem9}.
\begin{theorem}\label{theorem10}
Let $T \in B(H)$, where $H$ is a Hilbert space. If

(i) $\mbox{int} \thinspace \sigma_{p}(T)=\emptyset$,

(ii) $\mbox{int} \thinspace \sigma_{*}(T)=\emptyset$,

(iii) $\rho_{*}(T)$ consists of finite bounded components,\\
where $\sigma_{*}$,$\rho_{*}=\sigma_{sf}$, $\rho_{sf}$ or $\sigma_{sbf}$, $\rho_{sbf}$ or $\sigma_{qf}$, $\rho_{qf}$ or $\sigma_{\Gamma}$, $\rho_{\Gamma}$. Then there exists $\delta >0$ such that $T+K$ has SVEP for all $K \in K(H)$ with $\Vert K \Vert < \delta$. 
\end{theorem}
\begin{theorem}
Let $T \in B(H)$, where $H$ is Hilbert space. If $\sigma_{qf}(T)=\emptyset$, then  $$\sigma(T+K)=\mbox{iso} \thinspace \sigma_{sbf}(T+K) \cup \Pi(T+K)$$ for any compact operator $K \in K(X)$.
\end{theorem}
\begin{proof}
Since $\sigma_{qf}(T)=\emptyset$, $\rho_{qf}(T)=\mathbb{C}$ which implies that $\sigma(T)=\Pi(T)$. As $\mbox{int} \thinspace \sigma_{qf}(T)= \emptyset$, by Theorem \ref{theorem9} $\rho_{qf}(T+K)$ is connected. This gives $\rho_{qf}(T+K)=\rho_{sbf}(T+K)$. Therefore, by Theorem \ref{theorem11} $\sigma(T+K)=\sigma_{qf}(T+K) \cup \Pi(T+K).$ As $\sigma(T)= \Pi(T)$,  $\sigma(T)$ is finite which implies that $\sigma_{sf}(T)=\sigma_{sf}(T+K)$ is finite. This gives $$\sigma_{qf}(T+K)= \sigma_{sbf}(T+K)=\mbox{iso} \thinspace \sigma_{sbf}(T+K).$$ Therefore,  $\sigma(T+K)=\mbox{iso} \thinspace  \sigma_{sbf}(T+K) \cup \Pi(T+K).$
\end{proof}

\textbf{Acknowledgement} The corresponding author (Ankit Kumar) is supported by Department of Science and Technology, New Delhi, India (Grant No. DST/INSPIRE Fellowship/[IF170390]).

\textbf{Anuradha Gupta}\\
 Department of Mathematics, Delhi College of Arts and Commerce,\\
  University of Delhi, Netaji Nagar, \\
  New Delhi-110023, India.\\
  \vspace{0.2cm}
  email: dishna2@yahoo.in
  
  \textbf{Ankit Kumar$^*$}\\
  Department of Mathematics, University of Delhi, \\
  New Delhi-110007, India.\\
  email: 1995ankit13@gmail.com

\begin{thebibliography}{99}
\bibitem{8} P. Aiena, {\it Fredholm and local spectral theory, with applications to multipliers}, Kluwer Academic Publishers, Dordrecht, 2004. 

\bibitem{1} P. Aiena, {\it Fredholm and local spectral theory II}, Lecture Notes in Mathematics, 2235, Springer, Cham, 2018.

\bibitem{9} P. Aiena\ and\ S. Triolo, Weyl-type theorems on Banach spaces under compact perturbations, Mediterr. J. Math. {\bf 15} (2018), no.~3, Art. 126, 18 pp.

\bibitem{2} M. Berkani, \textit{On a class of quasi-Fredholm operators}, Integral Equations Operator Theory {\bf 34} (1999), no.~2, 244--249. 

\bibitem{3} M. Berkani\ and\ H. Zariouh, \textit{B-Fredholm spectra and Riesz perturbations}, Mat. Vesnik {\bf 67} (2015), no.~3, 155--165. 

\bibitem{10} B. P. Duggal\ and\ I. H. Kim, Generalized Browder, Weyl spectra and the polaroid property under compact perturbations, J. Korean Math. Soc. {\bf 54} (2017), no.~1, 281--302.

\bibitem{11} B. Jia\ and\ Y. Feng, Weyl type theorems under compact perturbations, Mediterr. J. Math. {\bf 15} (2018), no.~1, Art. 3, 13 pp.


\bibitem{4} W. Shi, \textit{Topological uniform descent and compact perturbations}, Rev. R. Acad. Cienc. Exactas F\'{\i}s. Nat. Ser. A Mat. RACSAM {\bf 113} (2019), no.~3, 2221--2233.

\bibitem{5} Q. Zeng, H. Zhong\ and\ Q. Jiang, \textit{Localized SVEP and the components of quasi-Fredholm resolvent set}, Glas. Mat. Ser. III {\bf 50(70)} (2015), no.~2, 429--440.

\bibitem{6} S. Zhu\ and\ C. G. Li, \textit{SVEP and compact perturbations}, J. Math. Anal. Appl. {\bf 380} (2011), no.~1, 69--75.

\bibitem{7} S. \v{Z}ivkovi{\'{c}}-Zlatanovi{\'{c}} 
and M. Berkani, \textit{Topological Uniform Descent, Quasi-Fredholmness and Operators Originated from Semi-B-Fredholm Theory}, Complex Analysis and Operator Theory, doi: 10.1007/s11785-019-00920-3.
\end{thebibliography}
\end{document}